\documentclass{amsart}
\usepackage{amsmath}

\usepackage{amssymb}
\usepackage{amsthm}
\usepackage{amscd,amsbsy}
\usepackage{mathtools}
\usepackage{xcolor}
\usepackage{tikz-cd}
\usetikzlibrary{arrows,matrix}
\usepackage{float}
\usepackage{mathtools}
\usepackage{enumitem}
\usepackage{mathrsfs}  
\usepackage{stmaryrd}

\newtheorem{theorem}{Theorem}[section]
\newtheorem{lemma}[theorem]{Lemma}

\newtheorem{corollary}[theorem]{Corollary}

\theoremstyle{definition}

\newtheorem{definitions and remarks}[theorem]{Definitions and Remarks}

\theoremstyle{remark}

\newtheorem{remark}[theorem]{Remark}

\numberwithin{equation}{section}

\newcommand{\inv}{\mathrm{inv}}
\newcommand{\ATWinv}{\mathrm{ATWinv}}

\newcommand{\supp}{\mathrm{supp}\,}

\newcommand{\ord}{\mathrm{ord}}

\newcommand{\nc}{\mathrm{nc}}

\newcommand{\quot}{\mathrm{quot}}

\newcommand{\lex}{\mathrm{lex}\,}

\newcommand{\al}{{\alpha}}
\newcommand{\be}{{\beta}}

\newcommand{\ep}{{\varepsilon}}

\newcommand{\ga}{{\gamma}}

\newcommand{\s}{{\sigma}}

\newcommand{\IN}{{\mathbb N}}

\newcommand{\IK}{{\mathbb K}}
\newcommand{\IL}{{\mathbb L}}

\newcommand{\IZ}{{\mathbb Z}}

\newcommand{\tb}{{\widetilde b}}

\newcommand{\tX}{{\widetilde X}}

\newcommand{\ts}{{\widetilde \sigma}}

\newcommand{\wxi}{{\widehat \xi}}

\newcommand{\wx}{{\widehat x}}

\newcommand{\llb}{{\llbracket}}
\newcommand{\rrb}{{\rrbracket}}
\newcommand{\llbr}{{(\!(}}
\newcommand{\rrbr}{{)\!)}}

\newcommand{\RN}[1]{%
  \textup{\uppercase\expandafter{\romannumeral#1}}%
}

\begin{document}
\title[Formal splitting and stack-theoretic nc desingularization]
{Formal splitting and stack-theoretic normal crossings desingularization}
\author[A.~Belotto da Silva]{Andr\'e Belotto da Silva}
\author[F.~Bernard]{Fran\c{c}ois Bernard}
\author[E.~Bierstone]{Edward Bierstone}
\address[A.~Belotto da Silva]{Universit\'e Paris Cit\'e and Sorbonne Universit\'e, UFR de Math\'ematiques, Institut de Math\'ematiques 
de Jussieu-Paris Rive Gauche, UMR7586, F-75013 Paris, France, and Institut universitaire de France (IUF)}
\email{andre.belotto@imj-prg.fr}
\address[F.~Bernard]{University of Toronto, Department of Mathematics, 40 St. George Street, Toronto, ON, Canada M5S 2E4}
\email{francois.bernard@utoronto.ca}
\address[E.~Bierstone]{University of Toronto, Department of Mathematics, 40 St. George Street, Toronto, ON, Canada M5S 2E4}
\email{bierston@math.utoronto.ca}
\thanks{Research supported by Agence Nationale de la
Recherche (ANR) projects ANR-25-ERCC-0003-01, ANR-22-CE40-0014 (Belotto da Silva), and
NSERC Discovery Grant RGPIN-2017-06537 (Bierstone, Bernard)}
\date{\today}


\begin{abstract}
We show that stack-theoretic resolution of singularities preserving normal crossings (partial desingularization)
by weighted blowings-up, can be obtained in a simple direct way from a splitting theorem of the first and third authors, using the algorithm
of Abramovich, Temkin and W{\l}odarczyk for resolution of singularities by weighted blowings-up.
\end{abstract}

\maketitle

\setcounter{tocdepth}{1}
\tableofcontents

\section{Introduction}\label{sec:intro}

The purpose of this note is to show that stack-theoretic resolution of singularities preserving normal crossings (\emph{partial desingularization})
by weighted blowings-up, can be obtained in a simple direct way from the splitting theorem \cite[Thm.\,1.1]{BB}, using the algorithm
of Abramovich, Temkin and W{\l}odarczyk \cite{ATW19} for resolution of singularities by weighted blowings-up.

\begin{theorem}\label{thm:pardes}
Let $X\subset Z$ ($Z$ smooth) denote an embedded variety over an uncountable algebraically closed field $\IK$ of characteristic zero.
Then, for any positive integer $k$, there is a stack-theoretic resolution of singularities $\ts: \tX \to X$ preserving the locus of normal crossings
points of $X$ of order $\leq k$; i.e.,
\begin{enumerate}
\item $\tX$ (as a stack) has only normal crossings points of order $\leq k$;

\smallskip
\item over the open subset of $X$ of normal crossings points of order $\leq k$, 
$\ts$ is \'etale and the associated modification $\s: X' = \tX_\quot \to X$ is an isomorphism.
\end{enumerate}
\end{theorem}

In particular, the quotient variety $X'$ has singularities in addition to normal crossings that are \'etale isomorphic to
quotients of normal crossings singularities by actions of finite abelian groups.

We are mainly interested in the case $k = \dim X +1$, so that $\s$ is an isomorphism over the entire normal crossings locus of $X$,
but the statement involving $k$ is useful for a proof by induction. 

We recall that a  \emph{normal crossings} singularity $\nc(k)$ (of \emph{order} $k$)
is defined by a monomial equation $x_1\cdots x_k = 0$, where $x_1,\ldots,x_k$ are formal or analytic
coordinates, as opposed to the more restrictive notion of \emph{simple normal crossings}, where $x_1,\ldots,x_k$ form part of
a system of regular parameters. (A normal crossings singularity at a given point of an algebraic variety $X$ is simple
normal crossings in an \'etale neighbourhood.)

The statement of Theorem \ref{thm:pardes} can be strengthened in various ways (e.g., functoriality, preservation of a simple normal crossings
divisor) or adapted to complex analytic varieties.
In particular, \cite{ABTW}, \cite{Wlodar} can be used in lieu of  \cite{ATW19} to account for a simple normal crossings divisor.
We do not go into
such generalizations because the focus here is on how to use the splitting theorem. The proof of the latter in \cite{BB} requires an
uncountable algebraically closed field.

The splitting theorem is used in \cite{BB} to obtain partial desingularization results where the aim is to give more
precise information on the nature of the quotient singularities. The article \cite{BB}
introduces a class of \emph{group-circulant singularities}
generalizing the classical Whitney umbrella, which cannot in general be eliminated from $X'$, when $\s: X' \to X$ is a proper
birational morphism preserving the normal crossings locus of $X$. See \cite[Thm.\,1.7]{BB}. Weighted blowing up of group
circulant singularities provides stack-theoretic normal crossings desingularization of $X$ \cite[Thm.\,1.13]{BB} (the latter is formulated in
the language of orbifolds). Theorem \ref{thm:pardes} above is essentially \cite[Thm.\,1.13]{BB} with only items (1)--(4) of the latter,
while item (5) deals with group-circulant singularities.

In this note, we show that Theorem \ref{thm:pardes} (or \cite[Theorem 1.13\,(1)--(4)]{BB} can be obtained as a simple consequence
of the splitting theorem (which itself is proved in less than three pages in \cite{BB}). 
We refer to \cite{BB} for the history of the problem and previous results.
Different approaches to Theorem \ref{thm:pardes} have been posted
by W{\l}odarczyk \cite{WlodarNC} and proposed by Abramovich and Temkin \cite{AT}.

\section{The splitting theorem}\label{sec:split}
Let $f$ denote a regular function on a smooth affine variety $Z$. Suppose $f$ has order
$k \in \IN$ on a smooth subvariety $S$ of $Z$, and $E$ is an snc (simple normal crossings)
divisor transverse to $S$.

Given $a_0 \in S$, there is an \'etale neighbourhood of $a_0$ in $Z$ with coordinates
\begin{equation}\label{eq:coords}
(w,u,x,z) = (w_1,\ldots,w_r, u_1,\ldots,u_s, x_1,\ldots,x_{k-1},z)
\end{equation}
in which $\{w_j=0\}$, $j=1,\ldots,r$, are the components of $E$ at $a_0=0$, $S = \{z=x=0\}$, and 
\begin{equation}\label{eq:weierpoly}
f(w,u,x,z) = z^k + a_2(w,u,x)z^{k-2} +\cdots + a_k(w,u,x),
\end{equation}
where the coefficients $a_i(w,u,x)$ are regular, and
$f$ is in the ideal generated by $x_1,\ldots,x_{k-1}, z$.

\begin{theorem}\label{thm:split}(See \cite[Thm.\,1.1]{BB}.)
Let $f(w,u,x,z)$ denote a function \eqref{eq:weierpoly}, where the coefficients $a_i(w,u,x)$ are regular,
$f$ is in the ideal generated by $x_1,\ldots,\allowbreak x_{k-1}, z$,
and $f$ splits formally (into $k$ factors of degree $1$ in $z$) at every point where $z=x=0,\, w_1\cdots w_r \neq 0$.
Assume that $\IK$ is uncountable.
Then, after finitely many blowings-up with successive centres of the form 
\begin{equation}\label{eq:splitcentre}
\{z=x=w_j =0\},\quad 1 \leq j \leq r,
\end{equation}
we can assume that $f$ splits over $\IK\llb w^{1/p},u,  x\rrb$, for some positive integer $p$,
where $w^{1/p} := (w_1^{1/p},\ldots, w_r^{1/p})$;
i.e., that 
\begin{equation}\label{eq:split}
f(w,u,x,z) = \prod_{i=1}^k \left(z + b_i (w, u, x)\right),
\end{equation}
where each $b_i (w, u, x) \in \IK\llb w^{1/p},u,  x\rrb$.
\end{theorem}

\begin{remark}\label{rem:split}
Let $S'$ denote the strict transform of $S := \{z=x=0\}$ by a blowing-up $\s$ with centre of the form \eqref{eq:splitcentre}.
Then $S' \cong S$; moreover, $S' = \{z=x=0\}$ again, in the $w_j$-coordinate chart of $\s$, and $a_0=0$ lifts to $a_0'=0$ in $S'$.
The statement in the theorem means that, after finitely many blowings-up with centres of the form \eqref{eq:splitcentre},
the strict transform $X'$ of $X$ is given at $a_0'=0$ by a function $f'$ with a splittting as in \eqref{eq:split}.
\end{remark}

In general, consider a regular function $f$ on $Z$, which has order $k$ on a smooth subvariety $S$. 

\begin{remark}\label{rem:codim2}
The \emph{nonsplitting locus}
$Y\subset S$ (the subset of $S$ of points where $f$ does not split formally into $k$ factors of order $1$) is a closed algebraic
subset of $S$ \cite[Lemma 2.2]{BB}. Moreover, if $Y$ has codimension $\geq 2$ in $S$, then
$Y=\emptyset$ (see the proof of \cite[Lemma 2.3]{BB}; this assertion is a simple consequence of 
\emph{purity of the branch locus} \cite{Nag1}, \cite{Nag2}, or Hartog's theorem in the 
analytic case.)
\end{remark}

Let $a_0 \in S$. Assume that $f$ can be written as in \eqref{eq:weierpoly} at $a_0$, with respect to coordinates
\eqref{eq:coords} such that $r=1$ and $S = \{z=x=0\}$ (so that $f$ is in the ideal generated by $x_1,\ldots,x_{k-1}, z$),
and that $f$ splits formally on $S\backslash\{w=0\}$. (For example, $a_0$ might be
a point of a \emph{codimension one stratum} $\{z=x=w_i=0;\, w_j\neq 0, j\neq i\}$ in the setting of the splitting theorem.)
We can apply the splitting theorem locally at $a_0$, with $r=1$. The conclusion of the theorem means that
(before blowing up) there are $p, q \in \IN$,
$p \geq 1$, such that $f$ splits formally at $a_0$ as in \eqref{eq:split}, where each 
\begin{equation}\label{eq:roots}
b_i (w, u, x) \in \IK\left\llbracket w^{1/p},u,  \frac{x}{w^q}\right\rrbracket.
\end{equation}

\medskip
Now suppose, in addition, that $f$ is generically $\nc(k)$ on $S$, and the desingularization invariant $\ATWinv$ of \cite{ATW19} 
(for the hypersurface defined by $f$) is
constant on $S$ (equivalently, the year zero desingularization invariant $\inv$ of \cite{BMinv}, \cite{BMfunct} is constant on $S$).

We will show it follows, then, that $f$ splits formally at $a_0$ with $p=1$ and $q=0$ in \eqref{eq:roots}, and 
$f$ is $\nc(k)$ at $a_0$.

Note that, if the nonsplitting locus $Y$ of $f$ were smooth and of codimension $1$ in $S$ at a point $a_0$, then we could choose
coordinates as above with $Y = \{z=x=w=0\}$, and the splitting with $p=1,\, q=0$ would contradict $a_0 \in Y$.
Therefore, $Y = \emptyset$ and it follows that $f$ is $\nc(k)$ on $S$.

Splitting with $p=1$ will be proved in Section \ref{sec:pardes}. We will show this is enough to prove the Partial Desingularization 
Theorem \ref{thm:pardes}.

Splitting with furthermore $q=0$ will be proved in Section \ref{sec:clopen}. (Reduction to $p=1$ and to $q=0$ can, in fact, be done in
either order.) As a corollary, we get the following
(cf. \cite[Lemma 4]{AT}, \cite[Thm.\,1.1.1]{WlodarNC}).

\begin{corollary}\label{cor:clopen}
Let $X\subset Z$ denote an embedded hypersurface and let $S$ denote the locally closed smooth subvariety
\begin{equation*}
S := \{\ATWinv = \ATWinv(\nc(k))\}
\end{equation*}
(i.e., $S$ is the subset of $X$ of points $a$ where $\ATWinv(a)$ equals its value at an $\nc(k)$ point).
Then the $\nc(k)$ locus of $X$ is open and closed in $S$.
\end{corollary}

We will also show that Theorem \ref{thm:pardes} follows directly from this corollary of the splitting theorem.

\section{The partial desgingularization theorem}\label{sec:pardes}
Following Section \ref{sec:split}, we consider a regular function $f$ on $Z$, and $S := \{\ATWinv = \ATWinv(\nc(k))\}$,
where $\ATWinv$ is the desingularization invariant of \cite{ATW19} for the hypersurface defined by $f$. In particular,
$f$ has order $k$ on $S$.

\begin{remark}\label{rem:homog} Suppose $S = \{z=x=0\}$, with respect to coordinates as in \eqref{eq:coords}.
Then, at any point $a\in S$, since $\ATWinv(a) = \ATWinv(\nc(k)) = (k,\ldots,k)$ ($k$ times), the lowest order homogeneous part of (the
formal expansion at $a$ of) $f$ is a homogeneous polynomial of degree $k$ in $(x,z)$, and $\ATWinv$ of this
homogeneous polynomial is also $\ATWinv(\nc(k))$.
\end{remark}

\begin{remark}\label{rem:assn} \emph{Assumptions and notation.}
Lemmas \ref{lem:negpower}--\ref{lem:p1} and \ref{lem:q0} below all use the following assumptions and notation.
Let $a_0 \in S$ and assume that $f$ is written as in \eqref{eq:weierpoly} at $a_0$, with respect to coordinates
\eqref{eq:coords} such that $r=1$ and $S = \{z=x=0\}$, and that $f$ is $\nc(k)$ on $S\backslash\{w=0\}$. As in Section
\ref{sec:split}, therefore,
there are $p, q \in \IN$, $p\geq 1$, such that $f$ splits formally at $a_0$ as in \eqref{eq:split}, with each 
$b_i (w, u, x) \in \IK\llb w^{1/p},u,  x/w^q\rrb$.

Let $\IL$ denote an algebraic closure $\overline{\IK(u)}$ of the field of fractions $\IK(u)$ of $\IK[u]$. Over $\IL$,
$\ATWinv(a_0)$ is also $\ATWinv(\nc(k))$. For each $i=1,\ldots,k$, we can write
\begin{equation}\label{eq:rootexp}
b_i(w,u,x) = \sum_{j=1}^{k-1} b_{ij}(w,u)x_j + O(x^2),
\end{equation}
where $O(x^2)$ means an element of the ideal $(x)^2$, and each $b_{ij}(w,u)$ is a formal Laurent series in
$w^{1/p}$ with coefficients in $\IL \cap \IK\llb u\rrb \subset \overline{\IK\llbr u\rrbr}$. Moreover, by the generic
$\nc(k)$ assumption, since $\sum b_i = 0$, for any given $i = i_0$, the matrix $\left(b_{ij}(w,u)\right)_{i\neq i_0}$
is invertible as a matrix with entries in $\overline{\IK(w^{1/p},u)}$.

It follows that, for each $i=1,\ldots,k$, we can write
\begin{equation*}
b_{ij}(w,u) = w^{d_i/p} \tb_{ij}(w,u),\quad j=1,\ldots,k-1,
\end{equation*}
where $d_i \in \IZ$, each $\tb_{ij}(w,u) \in \IK\llb w^{1/p},u\rrb$, and $\tb_{ij}(0,u)$ is not the zero power series, for some $j=j_i$.
\end{remark}

\begin{lemma}\label{lem:negpower}
$d_i \geq 0$, $i=1,\ldots,k$.
\end{lemma}

\begin{proof}
Consider $f$ and the $b_i$ over the field $\IL$. Each coefficient $a_j$ of $f$ is the elementary symmetric polynomial
of degree $j$ in the $b_i$. Therefore, $\ord\, a_j \geq j$, for all $j$, if and only if $\ord\, b_i \geq 1$, for all $i$,
where $\ord\, b_i$ is understood as the order with respect to $(w,x)$ of $b_i$ as a formal expansion in $x$ with coefficients
that are Laurent series in $w^{1/p}$.
The result follows.
\end{proof}

It follows from Lemma \ref{lem:negpower} that the lowest order homogeneous part of $f$ is
\begin{equation*}
\prod_{i=1}^k \bigg(z + \sum_{j=1}^{k-1} b_{ij}(0,0) x_j\bigg).
\end{equation*}
Since $\ATWinv(a_0) = \ATWinv(\nc(k)) = \ATWinv(\text{lowest order homogeoneous part})$ and $\sum b_i = 0$, we get the following lemma.

\begin{lemma}\label{lem:homog}
\begin{enumerate}
\item The matrix $\left(b_{ij}(0,0)\right)_{i\neq i_0}$ is invertible, for any choice of $i_0$.

\smallskip
\item The lowest order homogeneous part of $f$ is $\nc(k)$.

\smallskip
\item $d_i = 0,\, i=1,\ldots,k$; i.e., $b_{ij} = \tb_{ij}$, for all $i,j$.
\end{enumerate}
\end{lemma}

\begin{lemma}\label{lem:p1}
We can take $p=1$; i.e.,
each $b_i (w, u, x) \in \IK\llbracket w, u, x/w^q\rrbracket$.
\end{lemma}

\begin{proof}
Consider the smallest possible $p$, and assume that it is $>1$. There is an action of the multiplicative cyclic group $\mu_p$
on the set of roots $\{b_i\}$ induced by $\ep\cdot: w^{1/p} \to \ep w^{1/p}$, where $\ep = e^{2\pi i /p}$. Since $p>1$, the action of $\mu_p$
on the set of roots is nontrivial; i.e., there exists $i=i_0$ such that $\ep\cdot b_{i_0} = b_{i_1}$, where $i_1 \neq i_0$.
But then $b_{i_0,j}(0,u) = b_{i_1,j}(0,u)$, $j=1,\ldots,k-1$, in contradiction to Lemma \ref{lem:homog}(1).
\end{proof} 

\begin{proof}[Proof of the Partial Desingularization Theorem \ref{thm:pardes}]
Normal crossings singularities are hypersurface singularities. By ordinary resolution of singularities \cite{BMinv}, \cite{BMfunct},
there is a finite sequence of smooth blowings-up over the non-hypersurface points of $X$ (thus preserving nc points) after which
$X$ is a hypersurface.

We can therefore assume that $X\subset Z$ is an embedded hypersurface. We now argue by induction on $k$. The base
case $k=1$ is ordinary resolution of singularities. 

Given $k>1$, and following the algorithm of \cite{ATW19} for stack-theoretic desingularization by weighted blowings-up,
we can blow up until the maximum value of $\ATWinv$ is $\leq \ATWinv(\nc(k))$. Let
$S := \{\ATWinv=\ATWinv(\nc(k))\}$. Assume $S\neq \emptyset$. Then $S$ is the maximum locus of $\ATWinv$, so $S$ is 
a smooth closed subset of $X$. (We use the same notation $X, S$, etc., for the strict transforms of these objects by
the blowings-up involved.) Furthermore, we can blow up any component of $S$ on which $X$ is not generically $\nc(k)$,
so we can assume that $X$ is generically $\nc(k)$ on (every component of) $S$.

Let $Y\subset S$ denote the nonsplitting locus of $X$ in $S$ (i.e., of a local generator of the ideal of $X$ at any point of $S$).
By ordinary resolution of singularities, there is a sequence of smooth blowings-up with centres over $Y$,
after which $X$ is $\nc(k)$ on $S\backslash E$, where $E\subset Z$ is the exceptional divisor (transverse to $S$).

By the Splitting Theorem \ref{thm:split}, there is a positive integer $p$ and a finite sequence of blowings-up with 
centres $S\cap E_j$, where the $E_j$ are components of $E$, after which, at any point $a_0$ of $S$, we can choose
coordinates \eqref{eq:coords} in an \'etale neighbourhood of $a_0$, so that $\{w_j=0\}$, $j=1,\ldots,r$, are the components 
of $E$ at $a_0=0$, $S = \{z=x=0\}$, and the ideal of $X$ is generated by a function $f$ as in \eqref{eq:weierpoly},
which splits according to \eqref{eq:split} with roots $b_i (w, u, x) \in \IK\llb w^{1/p},u,  x\rrb$.

Note that each of the preceding two blow-up sequences preserves the property that $S := \{\ATWinv=\ATWinv(\nc(k))\}$,
since the blowings-up involved preserve the lowest order homogeneous part of (the formal expansion of) of a local
generator of the ideal of $X$.

If $a_0$ is a point of a codimension one stratum of $S\cap E$ (i.e., a point of $S$ lying in precisely one component of $E$),
then $r=1$ and the splitting occurs with $p=1$, by Lemma \ref{lem:p1}, so that $X$ is $\nc(k)$ at $a_0$. By 
Remark \ref{rem:codim2}, it follows that, for all $a\in S$, $X$ splits at $a$, and $X$ is $\nc(k)$ at $a$ since
$\ATWinv(a) = \ATWinv(\nc(k))$.

We can now apply the inductive
hypothesis in the complement of $S$. The centres of weighted blowing-up will be isolated from $S$ and, therefore,
closed in $X$ because $X$ is already normal crossings of order $<k$ in a deleted neighbourhood of $S$.
\end{proof}

\begin{remark}\label{rem:p1} 
The proof above uses Lemma \ref{lem:p1} only in the case that $q=0$. In this case, Lemma \ref{lem:negpower}
is tautological, so that Lemma \ref{lem:p1} is even simpler to prove.
\end{remark}

\section{Normal crossings strata}\label{sec:clopen}
Let us return to the setting of Section \ref{sec:pardes}. In particular, we use the notation and
assumptions of Remark \ref{rem:assn}, and we are assuming that $S := \{\ATWinv = \ATWinv(\nc(k))\}$,
$r=1$, and $f$ splits at $a_0$ as in \eqref{eq:split},
with each $b_i (w, u, x) \in \IK\llbracket w^{1/p}, u, x/w^q\rrbracket$.

\begin{lemma}\label{lem:q0}
Splitting \eqref{eq:split} holds with $q=0$; i.e., each $b_i (w, u, x) \in \IK\llbracket w^{1/p}, u, x\rrbracket$.
\end{lemma}

\begin{proof}
By Lemma \ref{lem:p1}, we can assume that $p=1$. Moreover, it is enough to prove Lemma \ref{lem:q0} in the
case $q=1$. Indeed, if we assume the lemma in this case, then splitting with each $b_i (w, u, x) \in \IK\llbracket w^{1/p}, u, x/w^q\rrbracket$,
for given $q>1$, implies that, after $q-1$ blowings-up with centre $\{z=x=w=0\}$, $f$ splits with $q=1$ and therefore with $q=0$,
by the assumption, so we also have $b_i (w, u, x) \in \IK\llbracket w^{1/p}, u, x/w^{q-1}\rrbracket$.

We can assume, therefore, that, \emph{a priori}, each $b_i (w, u, x) \in \IK\llbracket w, u, x/w\rrbracket$. (The preceding reduction
to this case is just
for the purpose of simplifying notation; the argument following can be carried out with general $p,q$.)

Recall the expansion \eqref{eq:rootexp}. By Lemma \ref{lem:homog}(1), we can make a change of coordinates to assume that
\begin{equation*}
b_i (w, u, x) = x_i + O(x^2),\quad i=1,\ldots,k-1.
\end{equation*}

Let us also write 
\begin{equation}\label{eq:rootexp1}
b_i (w, u, x) = \sum_{\al\in\IN^{k-1}} \sum_{\be \in \IZ} b_{i\al\be}(u) x^\al w^\be,\quad i=1,\ldots k,
\end{equation}
where each coefficient $b_{i\al\be}(u) \in \IK\llb u\rrb$. Let $\supp b_i := \{(\al,\be)\in \IN^{k-1}\times \IZ: b_{i\al\be}\neq 0\}$.

Order the elements $(\al,\be)=(\al_1,\ldots,\al_{k-1},\be) \in \IN^{k-1}\times \IZ$ using the lexicographic ordering
\begin{equation*}
\lex(|\al|+\be, |\al|, \al),
\end{equation*}
where $|\al| := \al_1 +\cdots + \al_{k-1}$.
Clearly, $(\al,\be) \geq \lex(0,0,0)$, for all $(\al,\be)\in \supp b_i$, $i=1,\ldots,k$, with respect to this ordering.

For each $i=1,\ldots,k$, let $(\al_i,\be_i)$ denote the smallest element of $\supp b_i$.

\medskip\noindent
\emph{Claim.} $(\al_i,\be_i) = (e_i,0),\, i=1,\ldots,k-1$, and $(\al_k,\be_k) = (e_{k-1},0)$, where $e_i := (0,\ldots,1,\ldots,0)$
(with $1$ in the $i$'th place and $0$ elsewhere).

\medskip
To prove this claim, first note that $(\al_i,\be_i) \leq (e_i,0),\, i=1,\ldots,k-1$, since $x_i$ is a monomial in the formal
expansion \eqref{eq:rootexp1}. Likewise, $(\al_k,\be_k) \leq (e_{k-1},0) < (e_i,0),\, i=1,\ldots,k-2$, since 
$b_k = -(b_1 +\cdots b_{k-1})$.

Suppose $(\al_\ell,\be_\ell) < (e_\ell,0)$, for some $\ell =1,\ldots,k-1$, or $(\al_\ell,\be_\ell) < (e_{k-1},0)$, for $\ell = k$.
This implies that $|\al_\ell| + \be_\ell < 1$, so that $\be_\ell < 0$ since $|\al_\ell| \geq 1$.

Now, $\prod_{i=1}^k b_i(w,u,x) = a_k(w,u,x) \in \IK\llb w,u,x\rrb$ and $\prod_i x^{\al_i} w^{\be_i}$ is a monomial in the formal expansion of the latter.
But, for all $i$, either $\be_i = 0$ or $\be_i <0$, and $\be_\ell < 0$, for some $\ell$, so that $w$ occurs to a negative
power in $\prod x^{\al_i} w^{\be_i}$; a contradiction. This proves the claim.

\medskip
It remains to show that $b_i (w, u, x) \in \IK\llbracket w, u, x\rrbracket$, $i=1,\ldots,k-1$. Suppose this is not true. Then
there is a monomial $x^\al w^\be \in \supp b_i$ with $\be<0$, for some $i=1,\ldots,k-1$.
Let $(\al_0,\be_0)$ denote the smallest $(\al,\be) \in \bigcup_{i=1}^{k-1} \supp b_i$ with $\be <0$.

For each $i=1,\ldots,k-1$, we can write
\begin{equation}\label{eq:roots2}
b_i(w,u,x) = x_i + b_{i\al_o\be_0}(u) x^{\al_0}w^{\be_0} + Q_i(w,u,x) + R_i(w,u,x),
\end{equation}
where 
\begin{enumerate}
\item $Q_i \in \IK\llbracket w, u, x\rrbracket$, and $|\al| + \be \geq 2$ for all $(\al,\be) \in \supp Q_i$;

\smallskip
\item $R_i \in \IK\llbracket w, u, x/w\rrbracket$ and, for every $(\al,\be) \in \supp R_i$, $\be < 0$ and 
$(\al_0,\be_0) < (\al,\be)$.
 \end{enumerate}
 
Let $h$ denote the smallest index $i=1,\ldots,k-1$ such that $b_{i\al_o\be_0}(u)$ is not the zero element of $\IK\llb u\rrb$;
i.e., such that $(\al_0,\be_0) \in \supp b_i$. Consider
the coefficient $a_{k-h+1}(w,u,x) \in \IK\llb w,u,x\rrb$ of $f$; of course, $a_{k-h+1} = \s_{k-h+1}(b_1,\ldots,b_k)$, where $\s_{k-h+1}$
denotes the elementary symmetric polynomial of order $k-h+1$.

We can use the identity 
\begin{multline*}
\s_{k-h+1}(\xi_1+y_1,\ldots,\xi_k +y_k)\\
= \s_{k-h+1}(\xi_1,\ldots,\xi_k) + \sum_{i=1}^k y_i\,\s_{k-h}(\xi_1,\ldots,\wxi_i,\ldots,\xi_k) +O(y^2),
\end{multline*}
setting $\xi_i = x_i,\, y_i = b_{i\al_o\be_0} x^{\al_0}w^{\be_0}$, for $i=1,\ldots,k-1$, and $\xi_k = -(\xi_1+\cdots +\xi_{k-1}),\,
y_k = -(y_1+\cdots +y_{k-1})$, to see that, in $\s_{k-h+1}(b_1,\ldots,b_k)$, the monomial $x^{\al_0}w^{\be_0}$ occurs, multiplied by
 \begin{align*}
 &\sum_{i=1}^{k-1}\left\{\s_{k-h}(x_1,\ldots,\wx_i,\ldots,x_{k-1},-(x_1+\cdots +x_{k-1})) - \s_{k-h}(x_1,\ldots, x_{k-1})\right\}\cdot b_{i\al_0\be_0}\\
 &\begin{multlined}[t]
 = \sum_{i=h}^{k-1}\left\{-(x_1+\cdots +x_{k-1}))\s_{k-h-1}(x_1,\ldots,\wx_i,\ldots,x_{k-1})\right.\\
  \left.\qquad\qquad+ \s_{k-h}(x_1,\ldots,\wx_i,\ldots,x_{k-1})
   -  \s_{k-h}(x_1,\ldots, x_{k-1})\right\}\cdot b_{i\al_0\be_0}\,.
   \end{multlined}
   \end{align*} 

Among the monomials $x^\al$ in the latter sum
(all of degree $|\al| = k-h$), the monomial
\begin{equation*}
x^{\ga_h} := x_{h+1}\dots x_{k-1}\cdot x_{k-1}
\end{equation*}
(or the monomial $x^{\ga_h} := x_{k-1}$, in the case $h=k-1$)
has the smallest exponent $\al$ (with respect to the lexicographic ordering of $\IN^{k-1}$), and
occurs uniquely in the summand where $i=h$.
Using the properties of \eqref{eq:roots2} above, we see that
\begin{equation*}
(\ga_h + \al_0,\,\be_0)\, \in\, \supp \s_{k-h+1}(b_1,\ldots,b_k)\,;
\end{equation*}
a contradiction.
\end{proof}

\begin{remark}\label{rem:pardes1}
As indicated in Section \ref{sec:split}, Corollary \ref{cor:clopen} is a consequence of Lemmas \ref{lem:p1} and \ref{lem:q0}.
Theorem \ref{thm:pardes} follows directly from Corollary \ref{cor:clopen}. Indeed, as in Section \ref{sec:pardes},
after blowing up using \cite{ATW19} until $S := \{\ATWinv=\ATWinv(\nc(k))\}$ is the
maximum locus of $\ATWinv$, it follows from Corollary \ref{cor:clopen} that $X$ is either $\nc(k)$ everywhere or $\nc(k)$ 
nowhere on each component of $S$. Then
we can again blow up to get rid of the components that are not $\nc(k)$, and simply finish by applying the
inductive hypothesis in the complement of $S$.

The purpose of the proof in Section \ref{sec:pardes} was to show how little is needed to deduce the Partial Desingularization 
Theorem \ref{thm:pardes} from the Splitting Theorem \ref{thm:split}.
\end{remark}

\bibliographystyle{amsplain}

\begin{thebibliography}{99}

\bibitem{ABTW}
D. Abramovich, A. Belotto da Silva, M. Temkin and J. W{\l}odarczyk,
\textit{Principalization on logarithmically foliated orbifolds},
preprint arXiv:2503.00926v2 [math.AG] (2025), 71 pages.

\bibitem{AT}
D. Abramovich and M. Temkin,
\textit{Partial desingularization in characteristic 0},
preprint arXiv:2602.15612v1 [math.AG] (2026), 4 pages.

\bibitem{ATW19}
D. Abramovich, M. Temkin and J. W{\l}odarczyk,
\textit{Functorial embedded resolution via weighted blowings up},
Alg. No. Theory \textbf{18} (2024), 1557--1587.

\bibitem{BB}
A. Belotto da Silva, and E. Bierstone,
\textit{Group-circulant singularities and partial desingularization preserving normal crossings},
preprint arXiv:2602.09114v1 [math.AG] (2026), 46 pages.

\bibitem{BMinv}
E. Bierstone and P.D. Milman,
\textit{Canonical desingularization in characteristic zero by
blowing up the maximum strata of a local invariant},
Invent. Math.
\textbf{128} (1997), 207--302.

\bibitem{BMfunct}
E. Bierstone and P.D. Milman,
\textit{Functoriality in resolution of singularities},
Publ. R.I.M.S. Kyoto Univ. \textbf{44} (2008), 609--639.

\bibitem{Nag1}
M. Nagata,
\textit{Remarks on a paper of Zariski on purity of branch-loci},
Proc. Natl. Acad. Sci. USA \textbf{44} (1958), 796--799.

\bibitem{Nag2}
M. Nagata,
\textit{On the purity of branch loci in regular local rings},
Illinois J. Math. \textbf{3} (1959), 328--333.

\bibitem{Wlodar}
J. W{\l}odarczyk,
\textit{Functorial resolution by torus actions}.
preprint arXiv:2203.03090v4 [math.AG] (2025), 50 pages.

\bibitem{WlodarNC}
J. W{\l}odarczyk,
\textit{Resolution except for the normal-crossing locus and Galois actions},
preprint arXiv:2602.14266v1 [math.AG] (2026), 38 pages.


\end{thebibliography}

\end{document}